\documentclass[12pt, twoside]{article}
\usepackage{amsmath,amsthm,amssymb}
\usepackage{times}
\usepackage{enumerate}

\def\titlerunning#1{\gdef\titrun{#1}}
\makeatletter
\def\author#1{\gdef\autrun{\def\and{\unskip, }#1}\gdef\@author{#1}}
\def\address#1{{\def\and{\\\hspace*{18pt}}\renewcommand{\thefootnote}{}%
\footnote {#1}}%
\markboth{\autrun}{\titrun}}
\makeatother
\def\email#1{e-mail: #1}
\def\subjclass#1{{\renewcommand{\thefootnote}{}%
\footnote{\emph{Mathematics Subject Classification (2010):} #1}}}
\def\keywords#1{\par\medskip
\noindent\textbf{Keywords.} #1}

\newtheorem{theorem}{Theorem}[section]

\newtheorem{definition}[theorem]{Definition}
\newtheorem{lemma}[theorem]{Lemma}
\newtheorem{corollary}[theorem]{Corollary}
\newtheorem{proposition}[theorem]{Proposition}

\newtheorem{remark}[theorem]{Remark}

\numberwithin{equation}{section}

\numberwithin{equation}{section}

\begin{document}

\baselineskip=17pt

\titlerunning{Existence results for second-order monotone differential inclusions}

\title{Existence results for second-order monotone differential inclusions on the positive half-line}

\author{Gheorghe Moro\c{s}anu}

\date{January 27, 2014}

\maketitle

\address{Gheorghe Moro\c{s}anu: Central European University, Department of Mathematics, Nador u. 9, 1051 Budapest, Hungary; \email{morosanug@ceu.hu}}

\subjclass{Primary 34G20; Secondary 34G25, 47J35}

\begin{abstract}
Consider in a real Hilbert space $H$ the differential equation
(inclusion) $(E)$: $p(t)u^{\prime \prime}(t)+q(t)u^{\prime}(t)\in
Au(t)+f(t)$ for a.a.  $t>0$, with the condition $(B)$: $u(0)=x \in
\overline{D(A)}$, where $A\colon D(A)\subset H\rightarrow H$ is a
(possibly set-valued) maximal monotone operator whose range
contains $0$; $p,q\in L^{\infty}(0,\infty )$, with $\mathrm{ess} \inf \ p>0$ and $q^+ \in L^1(0, \infty )$. 
More than four decades ago, V. Barbu established the existence of a
unique bounded (on $[0,\infty)$) solution to $(E)$, $(B)$, in the
particular case $p\equiv 1$, $q\equiv 0$ and $f\equiv 0$.
Subsequently the existence of bounded solutions in the homogeneous
case ($f\equiv 0$) has been further investigated by H. Brezis (1972),
N. Pavel (1976), L. V\'{e}ron (1974-76), and by E.I. Poffald and 
S. Reich (1984) when $A$ is an $m$-accretive operator in a Banach 
space. The non-homogeneous case
has received less attention from this point of view. R.E. Bruck
solved affirmatively this problem (in 1980), but under the
restrictive condition that $A$ is coercive (and $p \equiv 1$,
$q\equiv 0$, $f\in L^{\infty}(0,\infty ;H)$). On the other
hand, much attention has been paid by several authors 
to the asymptotic behavior of bounded solutions (if they exist) as $t\rightarrow \infty $,
both in the homogeneous and nonhomogeneous case. Recently, I
established jointly with 
H. Khatibzadeh [Set-Valued Var. Anal. 
DOI 10.1007/s11228-013-0270-3] the existence of (weak and strong) bounded solutions to
$(E)$, $(B)$, in the case $p \equiv 1$, $q\equiv 0$, under the
optimal condition $tf(t) \in L^1(0,\infty ; H)$. In this paper,
this result is extended to the general case of non-constant
functions $p, \, q$ satisfying the mild conditions above, thus
compensating for the lack of existence theory for such kind of
second order problems. 
Note that our results open up the possibility to apply Lions' method 
of artificial viscosity towards approximating the solutions of some 
nonlinear parabolic and hyperbolic problems, as shown in the last 
section of the paper. 

\keywords{Strong solution, weak solution, bounded solution, smoothing effect, minimization problem, the method of artificial viscosity}
\end{abstract}

\section{Introduction}
Let $H$ be a real Hilbert space with the inner product $(\cdot,
\cdot )$ and the induced norm $\Vert x \Vert = (x,x)^{1/2}$.
Consider the second-order differential equation (inclusion)
\begin{equation}\tag{$E$}\label{E}
p(t)u^{\prime \prime}(t)+q(t)u^{\prime}(t)\in Au(t)+f(t) \quad
\text{for a.a. $t\in {\mathbb{R}}_{+}: = [0,\infty)$,}
\end{equation}
with the condition
\begin{equation}\tag{$B$}\label{B}
u(0)=x \in \overline{D(A)},
\end{equation}
where

\noindent $(H1)$ \quad $A\colon D(A)\subset H\rightarrow H$ is a
(possibly set-valued) maximal monotone operator whose graph
contains $[0,0]$; \vskip2pt \noindent $(H2)$ \quad $p, \, q\in
L^{\infty}({\mathbb{R}}_{+}):= L^{\infty}({\mathbb{R}}_{+};
\mathbb{R})$, with $\mathrm{ess} \inf p>0$ and $q^+ \in
L^1({\mathbb{R}}_{+})$, where $q^+(t): = \max \, \{ q(t), 0\} $;
\vskip2pt\noindent and $f$ is a given $H$-valued function whose
(required) properties will be specified later. \vskip2pt\noindent
It is worth pointing out that in $(H1)$ one can assume without any
loss of generality that the range $R(A)$ of $A$ contains the null
vector, since this case reduces to the previous one for a maximal
monotone operator obtained from $A$ by shifting its domain.

Information on monotone operators can be found in \cite{VB3},
\cite{HB2}, \cite{GM}.

We continue with some historical comments:

It was V. Barbu who established for the first time the existence
of a unique bounded solution to equation \eqref{E} subjected to
\eqref{B}, in the special case $p\equiv 1$, $q\equiv 0$ and $f\equiv
0$, in \cite{VB1, VB2} (see also Chapter V in \cite{VB3}),
followed by the nice paper by H. Brezis \cite{HB1}, who considered
a more general condition at $t=0$; see also N. Pavel \cite{NP1}, 
as well as E.I. Poffald and S. Reich \cite{PR2} in the case 
when $A$ is an $m$-accretive operator in a Banch space. 
L. V\'{e}ron \cite{LV1, LV2} paid attention to the same problem 
(existence of
bounded solutions) in the case of (sufficiently smooth) variable
coefficients $p(t), \, q(t)$ and $f\equiv 0$. The existence of
bounded solutions in the non-homogeneous case (i.e., when $f$ is
not the null function) has received less attention. Recall that
Bruck \cite{Bruck1} established the existence of a bounded
solution on $\mathbb{R}$ of equation \eqref{E} (implying that all
solutions of \eqref{E} are bounded on ${\mathbb{R}}_+$), in the
case $p\equiv 1$, $q\equiv 0$, $f\in L^{\infty}(\mathbb{R} )$,
under the restrictive condition that $A$ is coercive. We also
mention the relatively recent article by Apreutesei \cite{NA},
addressing the case of sufficiently smooth coefficients $p, \, q$,
with $p(t)\ge p_0>0, \ q(t)\ge q_0>0$, and $x\in D(A)$.

On the other hand, there has been a great deal of work
pertaining to the asymptotic behavior of bounded solutions (if they exist) as
$t\rightarrow \infty$ of \eqref{E}, \eqref{B}, including the case
of periodic or almost periodic forcing. See \cite{GM1979, LV3,
Enzo1, Enzo2} for the case $p\equiv 1$, $q\equiv 0$, $f\equiv 0$.
The case $p\equiv 1$, $q\equiv 0$ and $f$ periodic or almost
periodic was thoroughly analyzed by Biroli \cite{MB1, MB2}, 
Bruck \cite{Bruck1, Bruck2}, and by Poffald and Reich \cite{PR1, PR2} 
in the case when $A$ is an $m$-accretive operator in a Banach space. 
In recent years Djafari Rouhani and
Khatibzadeh have established various results on the asymptotic
behavior of bounded solutions (if they exist), as $t\rightarrow \infty$, for both
constant and variable coefficients $p, \, q$, and for both the
homogeneous and non-homogeneous case of \eqref{E} (see
\cite{DK}-\cite{DK4}). \vskip10pt In order to compensate for the
lack of existence theory for such kind of second order problems, I
have recently started working on this subject. Recall that (as in
\cite{AP}) equation \eqref{E} can be written as
\begin{equation}\label{1.1}
(a(t)u^{\prime}(t))^{\prime} \in b(t)Au(t) + b(t)f(t),
\end{equation}
where
\[
a(t)=\exp \left(\int_{0}^{t} \frac{q(s)}{p(s)} ds\right), \quad
b(t)=\frac{a(t)}{p(t)}.
\]
Denote by $X$ the weighted space $L_{b}^{2}({\mathbb{R}}_{+}; H)=
L^{2}({\mathbb{R}}_{+}; H; b(t)dt)$, which is a real Hilbert space
with the scalar product
\[
{(f,g)}_X = \int_0^{\infty}b(t)(f(t),g(t)) dt,
\]
and the induced norm
\[
{\lVert f \rVert}^2_X = {(f,f)}_X.
\]
Recently \cite{GM2013, GM2014}, we proved the existence of a
unique strong solution $u\in X$ to equation $(E)$ subjected to
$u(0)=x\in \overline{D(A)}$, under the above conditions $(H1)$,
$(H2)$, where instead of $q^+\in L^1({\mathbb{R}}_+)$ we had
a different condition on $q$: either
$\mathrm{ess}\inf q>0$ or $\mathrm{ess}\sup q <0$. Note that there we
did replace the usual boundedness (on ${\mathbb{R}}_+$) condition
by a different one, namely $u\in X$, which may or may not imply
boundedness of $u$. More precisely, we proved that if $x\in D(A)$
and $f\in X$, then (cf. Theorem 3.1 in \cite{GM2013}), there
exists a unique $u\in X$, with $u^{\prime}, \, u^{\prime \prime}
\in X$, such that $u(0)=x$ and $u$ satisfies equation \eqref{E}
for a.a. $t > 0$. Since
\begin{align} \nonumber
a(t){\Vert u(t)\Vert}^2 =&{\Vert x \Vert}^2 + \int_0^t
\frac{d}{ds}\big( a(s)
{\Vert u(s)\Vert}^2\big)\, ds\\
\nonumber
 =&{\Vert x \Vert}^2 + \int_0^t qb{\Vert u(s)\Vert}^2\, ds + 2\int_0^ta(u,u^{\prime})\, ds \\
 \nonumber
\le&{\Vert x \Vert}^2 + M({\Vert u \Vert}_X^2 + {\Vert u \Vert}_X
{\Vert u^{\prime} \Vert}_X) < \infty,
\end{align}
it follows that
\begin{equation}\label{O}
\Vert u(t)\Vert =  O\Big(\exp
\big(-\frac{1}{2}\int_0^t\frac{q}{p}\, ds\big) \Big).
\end{equation}
If $x\in \overline{D(A)}$ and $f\in X$, then (cf. Theorem 1.1 in
\cite{GM2014}) there exists a unique $u\in C({\mathbb{R}}_+;H)\cap
X$, such that $u^{\prime}, \, u^{\prime \prime} \in
L^2_b([\varepsilon , \infty ); H)$ for all $\varepsilon >0$, such
that $u(0)=x$ and $u$ satisfies equation \eqref{E} for a.a. $t >
0$. Therefore \eqref{O} is again valid for $t \ge \varepsilon $.
So, if $\mathrm{ess}\inf q>0$, then $\Vert u(t)\Vert$ decays
exponentially to zero. In the case $\mathrm{ess}\sup q<0$ and
$f\in X$, $\Vert u(t) \Vert$ could be unbounded. This fact is
illustrated by the simple scalar equation $u^{\prime \prime} -
u^{\prime} = 1, \ t>0, \ u(0)=x$, that has a unique solution in
$X$ (here $X=L^2({\mathbb{R}}_+; e^{-t} dt)$), $u(t)=x-t$, which
is unbounded (and so are all the other solutions).

It is worth pointing out that the sign condition on $q$ (i.e.,
either $\mathrm{ess}\inf \ q>0$ or $\mathrm{ess} \sup \ q < 0$)
was essential in our previous treatment \cite{GM2013, GM2014}.
However, by an inspection of the proof of Theorem 1.1 in
\cite{GM2014} we can see that if in addition $A$ is strongly
monotone, then the existence of $u\in X$ follows in absence of the
sign condition on $q$. Of course, strong monotonicity is a very
restrictive condition. \vskip10pt Our aim in this paper is to
derive existence of bounded (on ${\mathbb{R}}_+$) solutions $u$ to
equation $(E)$ subjected to $u(0)=x\in \overline{D(A)}$, under
$(H1)$ and $(H2)$ above, including the alternative assumption
$q^+\in L^1({\mathbb{R}}_+)$ (which allows $q(t)$ to be "close" or
equal to zero), plus appropriate conditions on the nonhomogeneous
term $f$. Of course, replacing the condition $u\in X$ by a
boundedness one leads to a different problem that requires
separate analysis.

Note that our assumptions on $p$ and $q$ are weaker than those
previously used by other authors.

Concerning the methodology we use in this paper, note that, while
in \cite{GM2013, GM2014} we performed a global analysis within the
space $X$ defined above, here we derive existence on
${\mathbb{R}}_+$ (of bounded solutions, or bounded solutions in a
generalized sense, as specified below) by a limiting process
applied to a sequence of two-point boundary value problems on
$[0,n]$, $n=0,1,...$ This approach has some common features with
that used in \cite{KM} for the particular equation
\begin{equation}
u^{\prime \prime}(t) \in Au(t) + f(t), \ t>0, \label{simpler}
\end{equation}
subjected to $u(0)=x\in \overline{D(A)}$. Existence of
bounded solutions on ${\mathbb{R}}_+$ for equation \eqref{simpler}
in the case of a general maximal monotone $A$ was first
established in \cite{KM}. More precisely, in \cite{KM} a concept
of a weak solution was defined for equation \eqref{simpler}, and
the existence of a unique, bounded, weak solution $u=u(t), \ t\ge
0,$ was established under the optimal condition $tf(t)\in
L^1({\mathbb{R}}_+; H)$ (simple examples involving $A=0$ show
that this class of the $f$'s cannot be enlarged if we want to have 
bounded solutions); if, in addition, $f\in
L^2_{loc}([0, \infty );H)$, the solution $u$ of equation
\eqref{simpler} is strong (i.e., $u$ is twice differentiable and
satisfies \eqref{simpler} for a.a. $t>0$). In this paper we extend this
existence result to the case of variable coefficients $p$ and $q$
satisfying $(H2)$. In Theorems \ref{t1}, \ref{t2} we establish the
existence of weak and strong solutions $u$ to \eqref{E}, \eqref{B}
satisfying
\begin{equation}\tag{$C$}\label{C}
\sup_{t\ge 0} {a_-(t)}{\Vert u(t) \Vert}^2 < \infty ,
\end{equation}
where $a_-(t) = \exp{\Big( -\int_0^t q^-/p \Big)  }$,  $q^-(t)=-\min \, \{ q(t),0 \}$. If, in addition to $(H2)$, 
$q^- \in L^1({\mathbb{R}}_+)$ (i.e., $q\in L^1({\mathbb{R}}_+)\cap
L^{\infty}({\mathbb{R}}_+)$), then \eqref{C} becomes a real
boundedness condition, and (cf. Corollary \ref{corollary}) for
each pair $(x,f)\in \overline{D(A)} \times L^1({\mathbb{R}}_+; H; tdt)$
there exists a unique, weak, bounded solution $u$ of \eqref{E},
\eqref{B}. If in addition $f\in L^2_{loc}([0, \infty );H)$ then $u$ is even 
strong.

Note that we use a constructive method, suitable for the numerical
approximation and for the variational approach when $A$ is a
subdifferential operator. The smoothing effect on the starting values $x\in \overline{D(A)}$  is pointed
out. In the last section of the paper we show how our results can be 
used to approximate the solutions of some parabolic and hyperbolic 
problems by the method of artificial viscosity introduced by 
J.L. Lions \cite{JLL}.

\section{Some auxiliary results}
For a given $T\in (0,\infty)$ denote by $X_{T}$ the weighted space
$L^2(0,T;H; b(t)dt)$, where $b=b(t)$ is the function defined in
Section 1, restricted to the interval $[0,T]$. $X_{T}$ is a
Hilbert space equipped with the scalar product
$$
(f_1, f_2)_{X_{T}} = \int_0^T b(t)(f_1(t), f_2(t))\, dt,
$$
and the induced norm. In fact, under our conditions below, $X_{T}$
coincides with the usual $L^2(0,T;H)$ algebraically and
topologically. The need for the weight $b(t)$ will become obvious
in what follows.
\begin{proposition}\label{prop}
Assume that $p,\, q\in L^{\infty}(0,T)$, with $\mathrm{ess} \inf \
p>0$. Define $B:D(B)\subset X_{T} \rightarrow X_{T}$ by
$$
D(B):= \{ v\in X_{T}; \, v^{\prime}, \, v^{\prime \prime} \in
X_{T},\, v(0)=x, \, v(T)=y \},
$$
$$
Bv = -pv^{\prime \prime} -q v^{\prime},
$$
where $x, \, y\in H$ are given vectors. Then, $B$ is a maximal
monotone operator in $X_{T}$. More precisely, $B$ is the
subdifferential of the proper, convex, lower semicontinuous
function $\Psi : X_{T} \rightarrow [0,\infty )$  defined by
$$
\Psi (v) = \frac{1}{2} \int_0^T a(t)\Vert v^{\prime}(t)\Vert^2 dt
+ j(v(0)-x) + j(v(t)-y),
$$
where $j$ is the indicator function of the set $\{ 0\} \subset H$.
\end{proposition}

\begin{proof}
Note that  $Bv=-\frac{p}{a}(av^{\prime})^{\prime}$ for all $v\in
D(B)$, so the monotonicity of $B$ in $X_{T}$ (equipped with the
scalar product ${(\cdot , \cdot )}_{X_T}$ defined above) follows
easily. The rest of the proof is similar to the proof of Lemma 2.2
in \cite{GM2013}, so we omit it.
\end{proof}

\begin{lemma}\label{second}
Let $A$ satisfy $(H1)$, $p,\, q\in L^{\infty}(0,T)$, with
$\mathrm{ess} \inf \ p>0$, and let $f\in L^2(0,T;H)$. Then, for
all $x,\, y\in D(A)$, there exists a unique $u = u(t) \in
W^{2,2}(0,T;H)$ satisfying
\begin{equation} \label{boundedinterval}
p(t)u^{\prime \prime}(t) + q(t)u^{\prime}(t) \in Au(t) + f(t) \ \
\text{for a.a.} \ t\in (0,T),
\end{equation}
\begin{equation}\label{BC}
u(0)=x, \ u(T)=y.
\end{equation}
\end{lemma}

\begin{proof}
We can assume without any loss of generality that $y=0$
(otherwise, one can use the substitution $\tilde{u}(t) = u(t)
-y$). Let $A_{\lambda}$ be the Yosida approximation of $A$ for
$\lambda >0$, i.e., $A_{\lambda}z ={\lambda}^{-1}(z-J_{\lambda}z),
\ z\in H$, where $J_{\lambda} = (I + \lambda A)^{-1}$ (the
resolvent of $A$). Denote by $\bar{A}$ the realization of $A$ in
$X_{T}$, i.e., $\bar{A} = \{ [v,w]\in X_{T}\times X_{T}:\, [v(t),
w(t] \in A \ \text{for a.a.} \ t\in (0,T)\} $. Note that
${\bar{A}}_{\lambda} + B$ is maximal monotone in $X_{T}$ for all
$\lambda >0$, where $\bar{A}_{\lambda}$ is the Yosida
approximation of $\bar{A}$ and $B$ is the operator defined above
(see Proposition \ref{prop}), where $y=0$. Therefore, for each
${\lambda >0}$ there exists a $u_{\lambda}\in W^{2,2}(0,T;H)$ that
satisfies
\begin{equation}\label{lambda}
-pu_{\lambda}^{\prime \prime} - q u_{\lambda}^{\prime} +
A_{\lambda}u_{\lambda} + \lambda u_{\lambda} = -f \  \ a.e. \
\text{in} \ (0,T),
\end{equation}
\begin{equation}\label{BCs}
u_{\lambda}(0) = x, \ u_{\lambda}(T) = 0.
\end{equation}
Equation \eqref{lambda} can be equivalently written as
\begin{equation}\label{equiv}
{\big( au_{\lambda}^{\prime} \big)}^{\prime} =
b(A_{\lambda}u_{\lambda} + \lambda u_{\lambda} + f) \  \ a.e. \
\text{in} \ (0,T).
\end{equation}
If we multiply equation \eqref{equiv} by $u_{\lambda}(t)$,
integrate the resulting equation over $[\tau ,T]$, and use the
fact that $A_{\lambda}0=0$, we obtain
\begin{equation}\label{ineq}
\int_{\tau}^T\big( (au_{\lambda}^{\prime})^{\prime}, u_{\lambda}
\big) \, dt \ge \lambda \int_{\tau}^T b{\Vert
u_{\lambda}\Vert}^2dt + \int_{\tau}^Tb(f,u_{\lambda})\, dt,
\end{equation}
which implies (see \eqref{BCs})
\begin{equation}\label{zero}
-a(\tau)(u_{\lambda}^{\prime}(\tau),u_{\lambda}(\tau))-\int_{\tau}^Ta{\Vert
u_{\lambda}^{\prime}\Vert}^2dt \ge \lambda \int_{\tau}^T b{\Vert
u_{\lambda}\Vert}^2dt + \int_{\tau}^Tb(f,u_{\lambda})\, dt.
\end{equation}
Therefore
\begin{equation}
\frac{1}{2}a(\tau )\frac{d}{d\tau} {\Vert u_{\lambda}(\tau
)\Vert}^2 \le \int_{\tau}^Tb\Vert f\Vert \cdot \Vert u_{\lambda}
\Vert \, ds.
\end{equation}
Integrating this inequality over $[0,t]$ yields
\begin{eqnarray}
\frac{1}{2}a(t){\Vert u_{\lambda}(t)\Vert}^2 - \frac{1}{2}{\Vert
x\Vert}^2 -\frac{1}{2} \int_0^tbq{\Vert u_{\lambda}\Vert}^2d\tau
\le & \int_0^Td\tau \int_{\tau}^Tb\Vert f\Vert \cdot \Vert u_{\lambda} \Vert \, ds \nonumber \\
= & \int_0^T\tau b \Vert f\Vert \cdot \Vert u_{\lambda} \Vert \,
d\tau, \nonumber
\end{eqnarray}
which implies
\begin{equation} \nonumber
a(t){\Vert u_{\lambda}(t)\Vert}^2 + \int_0^t \frac{q^-}{p}a{\Vert
u_{\lambda}\Vert}^2d\tau \le
\end{equation}
\begin{equation} \label{qminus}
\big( {\Vert x \Vert}^2 + 2\int_0^T\tau b \Vert f\Vert \cdot \Vert
u_{\lambda} \Vert \, d\tau \big) +  \int_0^t \frac{q^+}{p}a{\Vert
u_{\lambda}\Vert}^2d\tau, \ 0\le t\le T.
\end{equation}
Recall that $q^+(t):=\max \ \{ q(t),0  \}$ and $q^{-}(t):= -\min \
\{q(t),0 \}$. It follows by the Gronwall-Bellman lemma that
\begin{equation}\label{hehe}
a(t){\Vert u_{\lambda}(t)\Vert}^2 \le \big( {\Vert x \Vert}^2 +
2\int_0^T\tau b \Vert f\Vert \cdot \Vert u_{\lambda} \Vert \,
d\tau \big)\exp\big( \int_0^t \frac{q^+}{p}\, d\tau\big), \ 0\le t
\le T,
\end{equation}
and so
\begin{equation}\label{lam}
{\Vert u_{\lambda}(t)\Vert}^2 \le M\big( {\Vert x \Vert}^2 +
2\int_0^T\tau b \Vert f\Vert \cdot \Vert u_{\lambda} \Vert \,
d\tau \big), \ 0\le t \le T,
\end{equation}
where $M=\exp{\big( \int_0^T \frac{q^-}{p}\, d\tau \big)}$.
Denoting $C_{\lambda} = \sup \{\Vert u_{\lambda}(t)\Vert : \, 0\le
t \le T  \}$, we obtain from \eqref{lam}
\begin{equation}
C_{\lambda}^2 \le M\big( {\Vert x \Vert}^2 +
2C_{\lambda}\int_0^T\tau b \Vert f\Vert  \, d\tau \big), \ 0\le t
\le T,
\end{equation}
which shows that $\sup_{\lambda >0} C_{\lambda}< \infty$, i.e.,
\begin{equation}\label{bdd}
\sup \{ \Vert u_{\lambda}(t)  \Vert ;\, 0\le t \le T, \ \lambda >0
\} < \infty .
\end{equation}
On the other hand, if we take $\tau =0$ in \eqref{zero}, we get
\begin{equation}\label{hoho}
\int_0^Ta{\Vert u_{\lambda}^{\prime}\Vert}^2 dt\le
-(u_{\lambda}^{\prime}(0), x) + \int_0^Tb\Vert f \Vert \cdot \Vert
u_{\lambda}\Vert \, dt,
\end{equation}
which implies (see also \eqref{bdd})
\begin{equation}\label{prime}
{\Vert u_{\lambda}^{\prime} \Vert }^2_{X_T} \le C_1\Vert
u_{\lambda}^{\prime}(0)\Vert +C_2,
\end{equation}
where $C_1, \, C_2$ are some positive constants. Now, multiplying
\eqref{equiv} by $A_{\lambda}u_{\lambda}$ and then integrating
over $[0,T]$, we obtain
\begin{equation}\label{Yo}
\int_0^T \Big( {\big( au_{\lambda}^{\prime}\big) }^{\prime},
A_{\lambda}u_{\lambda}\Big) \, dt\ge \int_0^T b{\Vert
A_{\lambda}u_{\lambda}\Vert }^2 dt + \int_0^T
b(f,A_{\lambda}u_{\lambda})\, dt.
\end{equation}
Integrating by parts in \eqref{Yo} leads to
\begin{equation}\label{Yos}
{\Vert A_{\lambda}u_{\lambda} \Vert}^2_{X_T} \le
-(u_{\lambda}^{\prime}(0), A_{\lambda}x) +{\Vert f
\Vert}_{X_T}{\Vert A_{\lambda}u_{\lambda} \Vert }_{X_T},
\end{equation}
since $ \big( u_{\lambda}^{\prime},
{(A_{\lambda}u_{\lambda})}^{\prime}\big) \ge 0$. Recall that
$\Vert A_{\lambda}x \Vert \le \Vert A^0x \Vert$ for all $\lambda
>0$, where $A^0$ is the minimal section of $A$. Therefore,
\eqref{Yos} implies
\begin{equation}\label{Yosi}
{\Vert A_{\lambda}u_{\lambda} \Vert }_{X_T}^2 \le C_3\Vert
u_{\lambda}^{\prime}(0)\Vert + C_4.
\end{equation}
By \eqref{lambda}, \eqref{bdd}, \eqref{prime}, and \eqref{Yosi} we
can see that
\begin{align}\nonumber
{\Vert u_{\lambda}^{\prime \prime} \Vert}_{X_T}^2 \le & C_5\big(
{\Vert u_{\lambda}^{\prime} \Vert }_{X_T}^2+\lambda {\Vert
u_{\lambda} \Vert }_{X_T}^2 + {\Vert A_{\lambda}u_{\lambda}  \Vert
}_{X_T}^2 + {\Vert f \Vert}_{X_T}^2\big) \\ \label{primeprime} \le
& C_6\Vert u_{\lambda}^{\prime}(0)\Vert + C_7,
\end{align}
for all $\lambda \in (0,\lambda_0 ]$, where $\lambda_0$ is an
arbitrarily   fixed positive number. On the other hand, using the
obvious relation
\begin{equation}
\int_0^T (T-t)u_{\lambda}^{\prime \prime}(t)\, dt =
-Tu_{\lambda}^{\prime}(0) -x,
\end{equation}
and \eqref{primeprime}, we derive
\begin{align}\nonumber
\Vert u_{\lambda}^{\prime}(0)\Vert  \le & C_8{\Vert
u_{\lambda}^{\prime \prime} \Vert}_{X_T} +C_9 \\ \nonumber \le &
C_8\sqrt{C_6\Vert u_{\lambda}^{\prime}(0)\Vert + C_7} + C_9 \\
\nonumber \le & C_{10}\sqrt{\Vert u_{\lambda}^{\prime}(0)\Vert} +
C_{11},
\end{align}
which shows that $\sup_{0< \lambda \le \lambda_0}\Vert
u_{\lambda}^{\prime}(0)\Vert < \infty$. So, according to
\eqref{prime}, \eqref{Yosi}, and \eqref{primeprime}, the sequences
$u_{\lambda}^{\prime}$, $u_{\lambda}^{\prime \prime}$,
$A_{\lambda}u_{\lambda}$ ($0< \lambda \le \lambda_0$) are all
bounded in $X_T$. Now, for $\lambda, \, \mu \in (0,  \lambda_0]$,
we derive from \eqref{equiv}
\begin{equation}\nonumber
\int_0^T\big( {\big( a(u_{\lambda}^{\prime} -
u_{\mu}^{\prime})\big) }^{\prime}, u_{\lambda}- u_{\mu}\big) \, dt
= \int_0^Tb(A_{\lambda}u_{\lambda} - A_{\mu}u_{\mu} + \lambda
u_{\lambda} - \mu u_{\mu}, u_{\lambda} - u_{\mu})\, dt,
\end{equation}
which implies
\begin{equation}\nonumber
-\int_0^Ta{\Vert u_{\lambda}^{\prime} - u_{\mu}^{\prime}
\Vert}^2dt = \int_0^Tb(A_{\lambda}u_{\lambda} - A_{\mu}u_{\mu},
J_{\lambda} u_{\lambda} - J_{\mu} u_{\mu})\, dt +
\end{equation}
\begin{equation}\label{long}
\int_0^Tb(A_{\lambda}u_{\lambda} - A_{\mu}u_{\mu},\lambda
A_{\lambda}u_{\lambda} - \mu A_{\mu}u_{\mu} )\, dt +\int_0^T b
(\lambda u_{\lambda} - \mu u_{\mu},u_{\lambda} - u_{\mu})\, dt.
\end{equation}
The first term of the right hand side of \eqref{long} is
nonegative since $A_{\lambda}u_{\lambda}(t) \in
AJ_{\lambda}u_{\lambda}(t)$, so by the information above we easily obtain
\begin{equation}
{\Vert u_{\lambda}^{\prime} - u_{\mu}^{\prime}  \Vert}_{X_T}^2 \le
C_{12} (\lambda + \mu ).
\end{equation}
We also have
\begin{align}\nonumber
\Vert u_{\lambda}(t) - u_{\mu}(t) \Vert = & \Vert \int_0^t (u_{\lambda}^{\prime} - u_{\mu}^{\prime})\, dt \Vert  \\
\le & C_{13} {\Vert u_{\lambda}^{\prime} - u_{\mu}^{\prime}
\Vert}_{X_T}, \ 0\le t \le T.
\end{align}
Therefore, there exists $u\in W^{2,2}(0,T; H)$, such that
$u_{\lambda} \rightarrow u$ in $C([0,T];H)$, $u_{\lambda}^{\prime}
\rightarrow u^{\prime}$ strongly in $X_T$, and
$u_{\lambda}^{\prime \prime} \rightarrow u^{\prime \prime}$ weakly
in $X_T$, as $\lambda \rightarrow 0^+$. Obviously, $u(0)=x$ and
$u(T)=0$. This $u$ is also a solution to equation
\eqref{boundedinterval}. Indeed, we can pass to the limit as
$\lambda \rightarrow 0^+$ in \eqref{lambda} regarded as an
equation in $X_T$. Note that $A_{\lambda}u_{\lambda}(t) \in
AJ_{\lambda}u_{\lambda}(t)$, for all $t\in [0,T]$, and
\begin{equation}\nonumber
{\Vert J_{\lambda}u_{\lambda} -u\Vert}_{X_T} \le \lambda{\Vert
A_{\lambda}u_{\lambda}\Vert}_{X_T} + {\Vert u_{\lambda} - u
\Vert}_{X_T} \rightarrow 0, \ \text{ as } \lambda \rightarrow 0^+,
\end{equation}
so by the demiclosedness of $\bar{A}$, the weak limit in $X_T$ of
$A_{\lambda}u_{\lambda}$ belongs to $\bar{A}u$, i.e.,
\begin{equation}\nonumber
-f+pu^{\prime \prime} + qu^{\prime} \in Au \ \text{ for a.a. }
t\in (0,T).
\end{equation}
It remains to prove that $u$ is unique. Let $v\in W^{2,2}(0,T;H)$
be another solution of problem \eqref{boundedinterval},
\eqref{BC}. We have
\begin{equation}\label{unique}
{\big( a(u^{\prime} - v^{\prime})\big)}^{\prime} \in b(Au - Av) \
\text{ for a.a. } t\in (0,T).
\end{equation}
Multiplying \eqref{unique} by $u(t) - v(t)$ and integrating over
$[0,T]$ we obtain
\begin{equation}\nonumber
-\int_0^Ta{\Vert u^{\prime} - v^{\prime} \Vert}^2dt \ge 0,
\end{equation}
which shows that $u^{\prime} - v^{\prime} \equiv 0$, i.e., $u-v$
is a constant function. Since $u(0)-v(0)=0$, it follows that
$u\equiv v$.
\end{proof}

\begin{remark}
For similar results we refer to \cite{AP}. Note that here $p$ and
$q$ satisfy weaker conditions. Note further that, not only the
conclusion of Lemma \ref{second}, but also some steps of its proof
will be used later.
\end{remark}

\begin{lemma}\label{third}
Assume that $A$ satisfies $(H1)$, $p,\, q\in L^{\infty}(0,T)$,
with $\mathrm{ess} \inf \ p>0$, $f\in L^2(0,T;H)$, and $x,\, y\in
D(A)$. For $\lambda >0$ denote by $u_{\lambda}$ the unique
solution of
\begin{equation} \label{bdd,lambda}
p(t)u_{\lambda}^{\prime \prime}(t) + q(t)u_{\lambda}^{\prime}(t) =
A_{\lambda}u_{\lambda}(t) + f(t) \ \ \text{for a.a.} \ t\in (0,T),
\end{equation}
\begin{equation}\label{BC,lambda}
u_{\lambda}(0)=x, \ u_{\lambda}(T)=y
\end{equation}
(which exists by Lemma \ref{second}). Then, $u_\lambda \rightarrow
u$ in $C([0,T];H)$ as $\lambda \rightarrow 0^+$, where $u$ is the
solution of problem \eqref{boundedinterval}, \eqref{BC}. Moreover,
$u_{\lambda}^{\prime} \rightarrow u^{\prime}$ in $C([0,T];H)$ and
$u_{\lambda}^{\prime \prime} \rightarrow u^{\prime \prime}$ weakly
in $X_T$, as $\lambda \rightarrow 0^+$.
\end{lemma}
\begin{proof}
Following a procedure similar to that used in the proof of Lemma
\ref{second}, we obtain $u_\lambda \rightarrow u$ in $C([0,T];H)$,
$u_{\lambda}^{\prime} \rightarrow u^{\prime}$ strongly in $X_T$,
and $u_{\lambda}^{\prime \prime} \rightarrow u^{\prime \prime}$
weakly in $X_T$, as $\lambda \rightarrow 0^+$. So actually
$u_{\lambda}^{\prime} \rightarrow u^{\prime}$ in $C([0,T];H)$ (by
Arzel\`{a}'s criterion).
\end{proof}

\section{Main Results}
We start this section by defining the concepts of strong and weak
solution for equation \eqref{E} (respectively, equation \eqref{E}
plus condition \eqref{B}) we shall use in what follows.

Note that in general we shall work under our assumptions $(H1)$
and $(H2)$ introduced in Section 1. For an interval $J\subset
\mathbb{R}$, open or not, denote by $L^p_{loc}(J;H)$
(resp. $W^{k,p}_{loc}(J;H)$) the space of all $H$-valued functions
defined on $J$, whose restrictions to compact intervals
$[a,b]\subset J$ belong to $L^p(a,b;H)$ (respectively, to
$W^{k,p}(a,b;H)$).
\begin{definition}
Let $f\in L^2_{loc}([0,\infty ); H)$ and let $x\in
\overline{D(A)}$. A $H$-valued function $u=u(t)$ is said to be a
{\bf strong solution} of equation \eqref{E} (respectively, of
equation \eqref{E} plus condition \eqref{B}) if $u\in C([0,\infty
); H)\cap W^{2,2}_{loc}((0,\infty ); H)$ and $u(t)$ satisfies
equation \eqref{E} for a.a. $t>0$ (and, in addition, $u(0)=x$,
respectively).
\end{definition}
Denote $Y=L^1(0,\infty ;H; t\sqrt{a_-(t)}dt)$, where $a_-(t)=
\exp{\Big( -\int_0^t\frac{q^{-}(\tau)}{p(\tau)}\, d\tau  \Big)}$.
Obviously, $Y$ is real Banach space with respect to the norm
$$
{\Vert f \Vert}_Y = \int_0^{\infty}\Vert f(t) \Vert
t\sqrt{a_-(t)}\, dt.
$$

If $f\in Y$ we cannot expect in general existence of strong
solutions for \eqref{E}, so we need the following definition.
\begin{definition}
Let $f\in Y$ and let $x\in \overline{D(A)}$. A $H$-valued function
$u=u(t)$ is said to be a {\bf weak solution} of equation \eqref{E}
(respectively, of equation \eqref{E} plus condition \eqref{B}) if
there exist sequences $u_n\in C([0,\infty ); H)\cap
W^{2,2}_{loc}((0,\infty ); H)$ and $f_n\in Y\cap
L^2_{loc}([0,\infty ); H)$, such that: $(i)$ $f_n$ converges to
$f$ in $Y$; $(ii)$   $u_n(t)$ satisfies equation \eqref{E} with
$f=f_n$ for a.a. $t>0$ and all $n\in \mathbb{N}$; and $(iii)$
$u_n$ converges uniformly to $u$ on any compact interval $[0,T]$
(and, in addition, $u(0)=x$, respectively).
\end{definition}
The concept of a weak solution for such second order differential
inclusions was previously introduced in \cite{KM} in the case
$p\equiv 1$ and $q\equiv 0$.

Note that the couple \eqref{E}, \eqref{B} is an incomplete
problem. While in \cite{GM2013}, \cite{GM2014} we added the
condition $u\in L^2({\mathbb{R}}_+ ; H; b(t)dt)$, in this paper we
consider a boundedness condition on ${\mathbb{R}}_+$: $\sqrt{a_-}{\Vert
u \Vert} \in L^{\infty}({\mathbb{R}}_+)$.

\begin{theorem}\label{t1}
Assume $(H1)$ and $(H2)$ hold. If $x\in \overline{D(A)}$, and
$f\in Y\cap L^2_{loc}([0,\infty ); H)$, then there exists a unique strong
solution $u$ of \eqref{E}, \eqref{B} which satisfies
\begin{equation}\tag{$C$}\label{C}
\sup_{t\ge 0} {a_-(t)}{\Vert u(t) \Vert}^2 < \infty . 
\end{equation}
Moreover, $\sqrt{ta_-}u^{\prime}
\in L^2({\mathbb{R}}_+;H)$ and $t^{3/2}u^{\prime \prime} \in
L^2_{loc}([0, \infty );H)$. If, in addition, $x\in D(A)$, then 
$u\in W^{2,2}_{loc}([0,\infty );H)$.
\end{theorem}

\begin{proof}
Let us assume in a first stage that $x\in D(A)$ (and $f\in Y\cap
L^2_{loc}([0,\infty ); H)$, as specified in the statement of the
theorem). For each $\lambda >0$ and $n\in \mathbb{N}$, denote by
$u_{n\lambda}, \, u_n$ the solutions of the following problems
\begin{equation}\label{unlambda}
pu_{n\lambda}^{\prime \prime} + qu_{n\lambda}^{\prime} =
A_{\lambda}u_{n\lambda} + f \ \text{ a.e. in } (o,n),
\end{equation}
\begin{equation}\label{BCunlambda}
u_{n\lambda}(0)=x, \ \ u_{n\lambda}(n)=0,
\end{equation}
and
\begin{equation}\label{un}
pu_{n}^{\prime \prime} + qu_{n}^{\prime} \in Au_{n} + f \ \text{
a.e. in } (o,n),
\end{equation}
\begin{equation}\label{BCun}
u_{n}(0)=x, \ \ u_{n}(n)=0.
\end{equation}
Lemma \ref{second} ensures the existence and uniqueness of
$u_{n\lambda}, \, u_n\in W^{2,2}(0,n;H)$. By a computation similar
to that performed in Lemma \ref{second} (see \eqref{hehe}), we get
\begin{equation}\label{inequ}
a_-(t){\Vert u_{n\lambda}(t)\Vert}^2 \le {\Vert x \Vert}^2 +
2\int_0^n\tau b \Vert f\Vert \cdot \Vert u_{n\lambda} \Vert \,
d\tau , \ 0\le t \le n.
\end{equation}
Denoting $M_{n\lambda}= {\sup}_{0\le t\le n}\sqrt{a_-(t)}\Vert 
u_{n\lambda}(t)\Vert $, we can derive from \eqref{inequ} the
following quadratic inequality
\begin{align}\label{Mnlambda} \nonumber
M_{n\lambda}^2 \le & \ {\Vert x \Vert}^2 +
2M_{n\lambda}\int_0^{\infty}\frac{\tau}{p} a_+ \sqrt{a_-} \Vert
f\Vert \, d\tau \\
\le & \ {\Vert x \Vert}^2 + 2\frac{a_+(\infty )}{p_0}{\Vert f \Vert}_YM_{n\lambda},
\end{align}
where $a_+(t) = \exp{ \big( \int_0^t\frac{q^+}{p} \, d\tau \big)}$ and $p_0 = \mathrm{ess} \inf \ p>0$, 
which shows that $M_{n\lambda}\le E =E(x,f):=D+\sqrt{D^2 + {\Vert
x \Vert}^2}$, $D:= \frac{a_+(\infty )}{p_0}{\Vert f \Vert}_Y$. Thus, 
\begin{equation}\label{bound_unlambda}
\sup_{0\le t \le n}{a_-(t)}{\Vert u_{n\lambda}(t) \Vert}^2 \le
E^2.
\end{equation}
Similarly, it follows from \eqref{un}, \eqref{BCun}
\begin{equation}\label{bound_un}
\sup_{0\le t \le n}{a_-(t)}{\Vert u_{n}(t) \Vert}^2 \le E^2.
\end{equation}
Now, let $0<R<m<n$, with $ m,n\in \mathbb{N}$. For a.a. $t\in
(0,m)$ we have
\begin{align} \label{ine} \nonumber
\frac{1}{2}\frac{d}{dt}\Big[ \, a\frac{d}{dt}{\Vert u_n -
u_m\Vert}^2  \, \Big] = & \frac{d}{dt}\big[ \, a(u_n^{\prime} -
u_m^{\prime}, u_n - u_m) \, \big] \\ \nonumber
= & \big( {\big( a(u_n^{\prime} - u_m^{\prime})\big)}^{\prime}, u_n - u_m\big)  + a {\Vert u_n^{\prime} - u_m^{\prime} \Vert}^2 \\
\ge &  a {\Vert u_n^{\prime} - u_m^{\prime} \Vert}^2.
\end{align}
By \eqref{ine} we have
\begin{align} \label{haha}\nonumber
\int_0^m(m-t)a(t){\Vert u_n^{\prime} - u_m^{\prime} \Vert}^2dt \le
\ & \frac{1}{2}\int_0^m(m-t)\frac{d}{dt}\big[ \,
a\frac{d}{dt}{\Vert u_n - u_m\Vert}^2  \, \big]\, dt \\ \nonumber
= \ & 0 + \frac{1}{2}\int_0^m a(t)\frac{d}{dt}{\Vert u_n -
u_m\Vert}^2dt \\ \nonumber
= \ & \frac{1}{2}a(t){\Vert u_n - u_m\Vert}^2\, |_0^m - \frac{1}{2}\int_0^m\frac{aq}{p}{\Vert u_n - u_m\Vert}^2dt \\
\le \ & \frac{1}{2}a(m){\Vert u_n(m) \Vert}^2
+\frac{1}{2}\int_0^m\frac{aq^-}{p}{\Vert u_n - u_m\Vert}^2dt.
\end{align}
As in the proof of Lemma \ref{second}, we can derive an inequality
for $u_k$ similar to \eqref{qminus} and therefore (see
\eqref{bound_un}) we have for all $t\in [0,k]$
\begin{align} \label{evrika}\nonumber
\int_0^t \frac{aq^-}{p}{\Vert u_{k}\Vert}^2d\tau \le \ &
\big( {\Vert x \Vert}^2 + 2\int_0^k\tau b \Vert f\Vert \cdot \Vert u_{k} \Vert \, d\tau \big) +  \int_0^t \frac{aq^+}{p}{\Vert u_{k}\Vert}^2d\tau \\
\le \ & \big( {\Vert x \Vert}^2 + 2DE \big) +  E^2a_+(\infty)\Big(
\int_0^{\infty}\frac{q^+}{p}\, d\tau \Big) :=F<\infty.
\end{align}
Combinig \eqref{bound_un}, \eqref{haha} and \eqref{evrika} leads
to
\begin{align}\label{good} \nonumber
(m-R)\int_0^Ra(t){\Vert u_n^{\prime} - u_m^{\prime} \Vert}^2dt \le
\ &
\int_0^m(m-t)a(t){\Vert u_n^{\prime} - u_m^{\prime} \Vert}^2dt \\
\le \ & \frac{1}{2}E^2a_+(\infty) + 2F.
\end{align}
This shows that $(u_n^{\prime})$ is a Cauchy sequence in $X_R =
L^2(0,R;H)$. Therefore $(u_n)$ is Cauchy in $C([0,R];H)$, since
\begin{equation}\nonumber
\Vert u_n(t) - u_m(t)\Vert = \Vert \int_0^t \big[
u_n^{\prime}(\tau) - u_m^{\prime}(\tau) \big] \, d\tau \Vert \le
Const.\, {\Vert u_n^{\prime} - u_m^{\prime} \Vert}_{X_R}, \ 0\le t
\le R.
\end{equation}
Since $R$ was arbitrarily chosen, it follows that there exists a
function $u\in C([0,\infty);H)$ $\cap W^{1,2}_{loc}([0,\infty );
H)$, such that $u_n\rightarrow u$ in $C([0,R];H)$ (so $u(0)=x$)
and $u_n^{\prime}\rightarrow u^{\prime}$ in $X_R$, for all $R>0$.
In addition, $a_-{\Vert u \Vert}^2 \in L^{\infty}({\mathbb{R}}_+)$
(cf. \eqref{bound_un}).

We intend to show that $u$ is a strong solution to equation $(E)$
by passing to the limit in equation \eqref{un} as $n\rightarrow
\infty$. To this purpose we  establish next a local $L^2$-estimate
for $u_n^{\prime \prime}$. By a resoning from the proof of Lemma
\ref{second} (see \eqref{hoho}), we have
\begin{equation}\label{hahaha}
\int_0^na{\Vert u_{n\lambda}^{\prime}\Vert}^2 dt\le
-(u_{\lambda}^{\prime}(0), x) + \int_0^nb\Vert f \Vert \cdot \Vert
u_{n\lambda}\Vert \, dt.
\end{equation}
For a given $R>0$, arbitrary but fixed, we derive from
\eqref{bound_unlambda} and \eqref{hahaha}
\begin{align}\label{hihi}\nonumber
\int_0^Ra{\Vert u_{n\lambda}^{\prime}\Vert}^2 dt\le & \ \Vert x
\Vert \cdot \Vert u_{n\lambda}^{\prime}(0)\Vert +  \Big( \int_0^R
+ \int_R^n \Big)b\Vert f \Vert \cdot \Vert u_{n\lambda}\Vert \, dt
\\ \nonumber \le & \ \Vert x \Vert \cdot \Vert
u_{n\lambda}^{\prime}(0)\Vert + E a_+(\infty ) \Big( \int_0^R
\frac{\sqrt{a_-}}{p}\Vert f \Vert \, dt +
\frac{1}{R}\int_R^{\infty}
\frac{t\sqrt{a_-}}{p}\Vert f\Vert \, dt \Big) \\
\le & \ \Vert x \Vert \cdot \Vert u_{n\lambda}^{\prime}(0)\Vert +
K_1,
\end{align}
where $K_1$ is a positive constant (depending on $R$, $x$, and
$f$). Now, multiplying the equation ${\big( au_{n\lambda}^{\prime}
\big)}^{\prime} = b(A_{\lambda}u_{n\lambda}+ f)$ by
$(R-t)^3A_{\lambda}u_{n\lambda}$ and integrating over $[0,R]$, we
obtain
\begin{align}\label{ah}\nonumber
\ \ & \int_0^R(R-t)^3b{\Vert A_{\lambda}u_{n\lambda}\Vert}^2dt \\
\nonumber \ \ = & \ \int_0^R\big(  {\big( au_{n\lambda}^{\prime}
\big)}^{\prime}, (R-t)^3 A_{\lambda}u_{n\lambda}  \big)\, dt
-\int_0^R(R-t)^3b(f, A_{\lambda}u_{n\lambda})\, dt \\ \nonumber \
\ = & \ -R^3(u_{n\lambda}^{\prime}(0), A_{\lambda}x) -
\int_0^R(R-t)^3a\big(
u_{n\lambda}^{\prime}, {\big( A_{\lambda}u_{n\lambda} \big)}^{\prime} \big) \, dt \\
\nonumber \ \ \ \ & + 3\int_0^R(R-t)^2a(u_{n\lambda}^{\prime},
A_{\lambda}u_{n\lambda})\, dt -\int_0^R(R-t)^3b(f,
A_{\lambda}u_{n\lambda})\, dt \\ \nonumber \ \  \le & \ R^3\Vert
A^0x\Vert\cdot \Vert u_{n\lambda}^{\prime}(0)\Vert -
\{\text{nonnegative term} \} \\ \nonumber \ \ \ \ &
+3\int_0^Ra\big( (R-t)^{\frac{1}{2}}u_{n\lambda}^{\prime},
(R-t)^{\frac{3}{2}} A_{\lambda}u_{n\lambda} \big)\, dt +
\int_0^R(R-t)^{\frac{3}{2} + \frac{3}{2}}b\Vert f \Vert \cdot
\Vert A_{\lambda}u_{n\lambda}\Vert \, dt \\ \nonumber \ \ \le & \
R^3\Vert A^0x\Vert\cdot \Vert u_{n\lambda}^{\prime}(0)\Vert + 3
{\Big( \int_0^R(R-t)a{\Vert u_{n\lambda}^{\prime}\Vert}^2
\Big)}^{\frac{1}{2}}
{\Big( \int_0^R(R-t)^3a{\Vert A_{\lambda}u_{n\lambda}\Vert}^2 \Big)}^{\frac{1}{2}} \\
\ \ \ \ & +{\Big( \int_0^R(R-t)^3b{\Vert
A_{\lambda}u_{n\lambda}\Vert}^2dt \Big)}^{\frac{1}{2}} {\Big(
\int_0^R(R-t)^3b{\Vert f \Vert}^2dt \Big)}^{\frac{1}{2}}.
\end{align}
By \eqref{hihi} and \eqref{ah} it follows that
\begin{equation}\label{hhh}
\int_0^R(R-t)^3b{\Vert A_{\lambda}u_{n\lambda}\Vert}^2dt \le
K_2\Vert u_{n\lambda}^{\prime}(0)\Vert + K_3.
\end{equation}
Denoting $Z_R=L^2(0,R/2;H)$, we derive from \eqref{hihi} and
\eqref{hhh}
\begin{equation}\label{ok}
{\Vert u_{n\lambda}^{\prime}\Vert}^2_{Z_R} \le K_4\Vert
u_{n\lambda}^{\prime}(0)\Vert + K_5, \ \ {\Vert
A_{\lambda}u_{n\lambda}\Vert}^2_{Z_R}\le K_6\Vert
u_{n\lambda}^{\prime}(0)\Vert + K_7,
\end{equation}
so that, according to \eqref{unlambda}, we also have
\begin{equation}\label{okok}
{\Vert u_{n\lambda}^{\prime \prime}\Vert}^2_{Z_R} \le K_8\Vert
u_{n\lambda}^{\prime}(0)\Vert + K_9.
\end{equation}
On the other hand,
\begin{equation}\nonumber
\frac{R}{2}u_{n\lambda}^{\prime}(0)= u_{n\lambda}(R/2) -x -
\int_0^{R/2} \Big( \frac{R}{2} -t \Big) u_{n\lambda}^{\prime
\prime}(t)\, dt,
\end{equation}
hence (see \eqref{bound_unlambda} and \eqref{okok})
\begin{align}\label{xxx}\nonumber
\Vert u_{n\lambda}^{\prime}(0)\Vert
\le & \ K_{10}{\Vert u_{n\lambda}^{\prime \prime}\Vert}_{Z_R} + K_{11} \\
\le & \ K_{12}\sqrt{\Vert u_{n\lambda}^{\prime}(0)\Vert} + K_{13},
\end{align}
which shows that $\Vert u_{n\lambda}^{\prime}(0)\Vert \le K_{14}$,
where $K_{14}$ is a constant depending on $R$, $x$, and $f$.
Therefore, according to \eqref{okok}, we have
\begin{equation}\label{aaar}
{\Vert u_{n\lambda}^{\prime \prime}\Vert}^2_{Z_R} \le K_{15}
:=K_8K_{14} + K_9.
\end{equation}
By Lemma \ref{third} we know that, for each $n\in \mathbb{N}$,
$u_{n\lambda}^{\prime \prime}$ converges weakly in $L^2(0,n;H)$
(hence in particular in $Z_R$) to $u_n^{\prime \prime}$ as
$\lambda \rightarrow 0^+$. This piece of information combined with
\eqref{aaar} shows that $(u_n^{\prime \prime})$ is bounded in
$Z_R$. Now, we are in  a position to take to the limit in
\eqref{un}, regarded as an equation in $Z_R$, to deduce that $u$
(the limit of $u_n$ in $C([0,R/2];H)$, hence in $Z_R$) belongs to
$W^{2,2}(0,R/2;H)$ and satisfies equation \eqref{E} for a.a. $t\in
(0,R/2)$. We have used the demiclosedness of the realization of
$A$ in $Z_R$ (see also the proof of Theorem 3.1 in \cite{KM}).
Since $R$ was arbitrarily chosen, this completes the proof of the
theorem in the case $x\in D(A)$.

Now we assume that $x\in \overline{D(A)}$ (and $f\in Y\cap
L^2_{loc}([0,\infty ); H)$, as specified in the statement of the
theorem). Let $x_k\in D(A)$ such that $\Vert x_k - x \Vert
\rightarrow 0$. For each $k$ denote by $u_k$ the strong solution
of \eqref{E} satisfying $u_k(0)=x_k$ and $\sqrt{a_-}\Vert u_k
\Vert \in L^{\infty}({\mathbb{R}}_+)$ (whose existence is ensured
by the previous part of the proof). For each $k$ let $u_{kn}$,
$u_{kn\lambda}$ be the corresponding approximations (see problems
\eqref{un}, \eqref{BCun} and \eqref{unlambda}, \eqref{BCunlambda}
above).

First, we have for a.a. $t\in (0,n)$
\begin{align} \label{inequal}
\frac{1}{2}\frac{d}{dt}\big[ \, a\frac{d}{dt}{\Vert u_{kn} -
u_{jn}\Vert}^2  \, \big] \ge & \  a {\Vert u_{kn}^{\prime} -
u_{jn}^{\prime} \Vert}^2,
\end{align}
which is similar to \eqref{ine} above. Hence the function
$t\rightarrow a(t)\frac{d}{dt}{\Vert u_{kn}(t) -
u_{jn}(t)\Vert}^2$ is nondecreasing on $[0,n]$. Since it equals
zero at $t=n$, it follows that it is $\le 0$ for all $t\in [0,n]$,
hence the function $t \rightarrow \Vert u_{kn}(t) -
u_{jn}(t)\Vert$ is nonincreasing on $[0,n]$. In particular,
\begin{equation}\label{kjn}
\Vert u_{kn}(t) - u_{jn}(t)\Vert \le \Vert x_k - x_j \Vert \ \
\forall t\in [0,n].
\end{equation}
Therefore,
\begin{equation}\label{kj}
\Vert u_{k}(t) - u_{j}(t)\Vert \le \Vert x_k - x_j \Vert \ \
\forall t\ge 0,
\end{equation}
which shows that there exists a function $u\in C([0,\infty );H)$
such that $u_k$ converges to $u$ in $C([0,R];H)$ for all $R\in (0,
\infty )$, so in particular $u(0)=x$. According to
\eqref{bound_un} (where $E = E(x_k,f)$ is bounded), we also have
 $\sqrt{a_-}{\Vert u \Vert} \in L^{\infty}({\mathbb{R}}_+)$.

On the other hand, we have for a.a. $t\in (0,n)$
\begin{align}\label{huhuhu}\nonumber
\frac{1}{2}\frac{d}{dt}\big[ a(t)\frac{d}{dt}{\Vert u_{kn}(t)
\Vert }^2 \big] = & \ \frac{d}{dt}\big( a{u^{\prime}_{kn}},u_{kn}
\big)
\\ \nonumber
 = & \Big( {\big( a{u^{\prime}_{kn}} \big)
}^{\prime}, u_{kn} \Big) + a{\Vert u_{kn}^{\prime} \Vert}^2
\\ \nonumber
\ge & \ b(f,u_{kn}) + a{\Vert u_{kn}^{\prime} \Vert}^2 \\
\nonumber \ge & \ -b\Vert f\Vert \cdot \Vert u_{kn} \Vert +
a{\Vert u_{kn}^{\prime} \Vert}^2 \\ \nonumber \ge & \
-\frac{1}{p_0}\big( \sqrt{a_-}\Vert f \Vert  \big) \cdot \big(
\sqrt{a_-}\Vert u_{kn}
\Vert  \big) + a{\Vert u_{kn}^{\prime} \Vert}^2 \\
\ge & \ -\frac{E(x_k,f)}{p_0}\big( \sqrt{a_-}\Vert f \Vert  \big)
+ a{\Vert u_{kn}^{\prime} \Vert}^2,
\end{align}
where $p_0 =\mathrm{ess} \inf \ p>0$. To obtain the last
inequality we have used \eqref{bound_un}. Now we multiply
\eqref{huhuhu} by $t$ and then integrate over $[0,n]$ to derive
\begin{align}\label{sasa}\nonumber
\int_0^nta{\Vert u_{kn}^{\prime} \Vert}^2dt \le & \ M
-\frac{1}{2}\int_0^na\frac{d}{dt}{\Vert u_{kn} \Vert}^2dt
\\ \nonumber
\le & \ M + \frac{1}{2}{\Vert x_k \Vert}^2 +
\int_0^n\frac{aq}{p}{\Vert u_{kn} \Vert}^2dt \\ \nonumber \le & \
M + \frac{1}{2}{\Vert x_k \Vert}^2 +
\frac{1}{p_0}\int_0^na_+q^+\big( a_-{\Vert u_{kn} \Vert}^2 \big)
\, dt \\ \nonumber \le & \ M + \frac{1}{2}{\Vert x_k \Vert}^2 +
\frac{1}{p_0}E(x_k,f)^2a_+(\infty){\Vert q^+
\Vert}_{L^1({\mathbb{R}}_+)} \\
\le & \ M_1,
\end{align}
where $M$ and $M_1$ are finite constants, since $\Vert x_k \Vert$
and $E(x_k,f)$ are bounded sequences. Therefore,
\begin{equation}\label{mmm}
\int_0^{\infty}ta{\Vert u_{k}^{\prime} \Vert}^2dt \le M_1,
\end{equation}
i.e., the sequence $(\sqrt{ta_-}u^{\prime}_k)$ is bounded in
$L^2({\mathbb{R}}_+; H)$. In fact this sequence is convergent in
$L^2({\mathbb{R}}_+; H)$ (and so its limit $\sqrt{ta_-}u^{\prime}
\in L^2({\mathbb{R}}_+; H)$). Indeed, multiplying \eqref{inequal}
by $t$ and then integrating the resulting inequality over $[0,n]$,
we get
\begin{align}\label{da} \nonumber
\int_0^n ta(t){\Vert u_{kn}^{\prime} - u_{jn}^{\prime} \Vert}^2dt
\le & \
  -\frac{1}{2}\int_0^na(t)\frac{d}{dt}{\Vert u_{kn}^{\prime} - u_{jn}^{\prime} \Vert}^2dt \\ \nonumber
= & \ \frac{1}{2}{\Vert x_k - x_j \Vert}^2 + \frac{1}{2}\int_0^n
\frac{aq}{p}{\Vert u_{kn} - u_{jn} \Vert}^2dt \\ \nonumber
\le & \ \frac{1}{2}{\Vert x_k - x_j \Vert}^2  + \frac{1}{2}\int_0^n\frac{a_+ q^+}{p}{\Vert u_{kn} - u_{jn} \Vert}^2dt \\
\le & \ \frac{1}{2} \Big[  1 + \frac{1}{p_0}a_+(\infty ) {\Vert
q^+ \Vert}_{L^1({\mathbb{R}}_+)} \Big] {\Vert x_k - x_j \Vert}^2.
\end{align}
To derive the last inequality we have used inequality \eqref{kjn}.
From \eqref{da} it follows that
\begin{equation}\label{der}
\int_0^{\infty} ta(t){\Vert u_{k}^{\prime} - u_{j}^{\prime}
\Vert}^2dt \le  \ c{\Vert x_k - x_j \Vert}^2,
\end{equation}
where $c =\frac{1}{2} \Big[  1 + \frac{1}{p_0}a_+(\infty ) {\Vert
q^+ \Vert}_{L^1({\mathbb{R}}_+)} \Big]$, which shows that
$(\sqrt{ta_-}u^{\prime}_k)$ is Cauchy, hence convergent, in
$L^2({\mathbb{R}}_+; H)$, as asserted.

In the following we shall prove that $t^{3/2}u^{\prime \prime} \in
L^2_{loc}([0,\infty ); H)$ and $u$ is a strong solution of
equation \eqref{E}. To this purpose, it is enough to establish a
local $L^2$-estimate for $u_k^{\prime \prime}$. We need to use
$u_{kn\lambda}$, the solution of
\begin{equation}\label{fox}
{\big( au^{\prime}_{kn\lambda} \big)}^{\prime} =
b(A_{\lambda}u_{kn\lambda} + f) \ \text{ a.e. in } \ (0,n); \
u_{kn\lambda}(0)=x_k, \ u_{kn\lambda}(n)=0.
\end{equation}
For a given $R>0$, $R<n$, define $h_R(t)=\min \{t, R-t \} , \ 0\le
t \le R$. Multiplying equation \eqref{fox} by
$h_R^3(t)A_{\lambda}u_{kn\lambda}(t)$ and integrating over
$[0,R]$, we get the following inequality (which is similar to
\eqref{ah} above)
\begin{align}\label{ahah}\nonumber
\int_0^Rh_R^3b{\Vert A_{\lambda}u_{kn\lambda}\Vert}^2dt \le  &
-3\int_0^Rh_R^2h_R^{\prime}\big( u_{kn\lambda}^{\prime},
 A_{\lambda}u_{kn\lambda}  \big) \, dt -\int_0^Rh_R^3b
\big(  f, A_{\lambda}u_{kn\lambda} \big) \, dt \\ \nonumber \le &
\ 3\int_0^Rah_R^{3/2}\Vert A_{\lambda}u_{kn\lambda} \Vert \cdot
h_R^{1/2} \Vert u_{kn\lambda}^{\prime}  \Vert \, dt \\
& \ + \int_0^Rh_R^3b \Vert f \Vert \cdot \Vert
A_{\lambda}u_{kn\lambda} \Vert \, dt.
\end{align}
Here, following an idea from \cite{Bruck2}, we have used the
function $h_R$ in order to get rid of the
term involving $A_{\lambda}x_k$ which is no longer bounded. Of
course, the information we get is a bit weaker, but enough to
ensure that $u$ is a strong solution. Arguing as before (see
\eqref{sasa}), we find that
\begin{equation}\nonumber
\int_0^Rta{\Vert u^{\prime}_{kn\lambda}  \Vert}^2dt \le M_1,
\end{equation}
where $M_1$ is the same constant as in \eqref{sasa}. It follows
that
\begin{equation}\label{aaa}
\int_0^Rh_Ra{\Vert u^{\prime}_{kn\lambda}  \Vert}^2dt \le M_1.
\end{equation}
By \eqref{ahah} and \eqref{aaa} we see that $\{
t^{3/2}A_{\lambda}u_{kn\lambda}; \ \lambda >0 \}$ is uniformly
bounded in $L^2(0,R/2;H)$ and so is $\{ t^{3/2}u^{\prime
\prime}_{kn\lambda}; \ \lambda >0 \}$ (by using the  equation $
pu_{kn\lambda}^{\prime \prime} + q u_{kn\lambda}^{\prime} =
A_{\lambda}u_{kn\lambda} + f$). Consequently, the sequence
$(t^{3/2}u^{\prime \prime}_{k})$ is also bounded in
$L^2(0,R/2;H)$. For a small $\delta >0$, denote $Z_{\delta, R} =
L^2(\delta , R/2;H)$. In this space, $u_k$ converges strongly to
$u$, $u_k^{\prime}$ converges strongly to $u^{\prime}$ (cf.
\eqref{der}), while $u^{\prime \prime}_{k}$ converges weakly to
$u^{\prime \prime}$ (in fact, $t^{3/2}u^{\prime \prime} \in
L^2(0,R/2;H)$). Passing to the limit in the equation
$$
pu_k^{\prime \prime} + q u_k^{\prime}\in Au_k + f,
$$
regarded as an equation in $Z_{\delta, R}$, we see that $u$
satisfies equation \eqref{E} for a.a. $t\in (\delta , R/2)$, hence
for a.a. $t>0$, since $\delta$ and $R$ were arbitrarily chosen.

To complete the proof, let us show that $u$ is unique. Assume
that $v$ is another strong solution of \eqref{E}, \eqref{B}
satisfying $\sqrt{a_-}\Vert v \Vert \in L^{\infty}({\mathbb{R}}_+)$. 
Then, for a.a. $t>0$,
\begin{align}\label{gaga}\nonumber
\frac{1}{2}\frac{d}{dt}\Big[ a(t)\frac{d}{dt}{\Vert u(t) - v(t)
\Vert}^2 \Big] = & \ \frac{d}{dt}\Big( a(u^{\prime} -
v^{\prime}),u-v \Big) \\ \nonumber = & \ \Big( {\big( a(u^{\prime} -
v^{\prime}) \big)}^{\prime}, u-v  \Big) + a{\Vert u^{\prime} -
v^{\prime} \Vert}^2 \\ \nonumber
\ge & \ a{\Vert u^{\prime} - v^{\prime} \Vert}^2 \\
\ge & \ 0.
\end{align}
This shows that the function $t\rightarrow a(t)\frac{d}{dt}{\Vert
u(t) - v(t) \Vert}^2$ is nondecreasing on $[0, \infty )$. It is
also nonnegative, since it vanishes at $t=0$. Hence, for $0<t<s$ we have 
\begin{equation}\label{www}
0\le a(t)\frac{d}{dt}{\Vert u(t)-v(t) \Vert}^2 \le a(s)\frac{d}{ds}{\Vert u(s)-v(s) \Vert}^2.
\end{equation}
Integrating \eqref{www} with respect to $s$ over $[t,T]$ yields
\begin{align}\label{vvu} \nonumber 
(T-t)a(t)\frac{d}{dt}{\Vert u(t)-v(t) \Vert}^2 \le & \ a(T){\Vert u(T)-v(T)\Vert}^2 - 
\int_t^T \frac{q(s)}{p(s)}a(s){\Vert u(s)-v(s) \Vert}^2 ds \\ \nonumber
\le & \ K + K\int_t^T\frac{q^-(s)}{p(s)}\, ds \\
\le & \ K + M(T-t),
\end{align}
where $K$ is a positive constant (depending on $u$ and $v$), and 
$M=K\cdot \mathrm{ess} \sup \frac{q^-}{p} < \infty $. Dividing \eqref{vvu} 
by $T-t$ and letting $T \rightarrow \infty $, we get
\begin{equation}\label{bbb}
0\le a(t)\frac{d}{dt}{\Vert u(t)-v(t) \Vert}^2 \le M \ \ \forall t\ge 0.
\end{equation} 
Combining \eqref{gaga} and \eqref{bbb} leads to 
\begin{equation}\nonumber
\int_0^{\infty } a(t){\Vert u^{\prime}(t) - v^{\prime}{t}\Vert}^2 dt \le \frac{M}{2}, 
\end{equation}
which implies 
\begin{equation}\label{uuu}
\liminf_{t\rightarrow \infty } \sqrt{a(t)}\Vert u^{\prime}(t) - v^{\prime}(t)\Vert =0. 
\end{equation}
Since 
\begin{align} \nonumber
0\le a(t)\frac{d}{dt}{\Vert u(t)-v(t) \Vert}^2 = & \ 2a(t)(u^{\prime}(t) - 
v^{\prime}(t), u(t)-v(t)) \\ \nonumber
\le & \ 2\sqrt{K}\sqrt{a(t)}\Vert u^{\prime}(t) - v^{\prime}(t) \Vert, 
\end{align}
it follows by \eqref{uuu} that 
\begin{equation}\nonumber
\lim_{t\rightarrow \infty } a(t)\frac{d}{dt}{\Vert u(t) - v(t)
\Vert}^2 =0,
\end{equation}
which clearly implies that $a(t)\frac{d}{dt}{\Vert u(t) - v(t)
\Vert} = 0 \ \forall t\ge 0$. Hence $\Vert u(t)-v(t) \Vert = \Vert u(0)-v(0) \Vert =0 \ \forall
t\ge 0$.
\end{proof}

\begin{theorem}\label{t2}
Assume $(H1)$ and $(H2)$ hold. Then, for each $x\in
\overline{D(A)}$ and $f\in Y$, there exists a unique weak solution
$u$ of \eqref{E}, \eqref{B}, \eqref{C}, and  $\sqrt{ta_-}u^{\prime} \in
L^2({\mathbb{R}}_+;H)$.
\end{theorem}

\begin{proof}
Let $x\in \overline{D(A)}$ and let $f_1, \, f_2\in Y\cap
L^2_{loc}([0,\infty );H)$. Denote by $u(t,x,f_i)$, $i=1,2$, the
corresponding strong solutions given by Theorem \ref{t1}, and by
$u_n(t,x,f_i)$ their approximations ($i=1,2, \ n\in
\mathbb{N}$), as defined above (see \eqref{un} and \eqref{BCun}).
It is easily seen that for a.a. $t\in (0,n)$
\begin{align}\label{vvv}\nonumber
\frac{1}{2}\frac{d}{dt}\Big[ a(t)\frac{d}{dt}&{\Vert u_n(t;x,f_1) - u_n(t; x,f_2)\Vert}^2 \Big]  \\
& \ge \ -\ b(t)\Vert f_1(t) - f_2(t)\Vert \cdot \Vert u_n(t;x,f_1)
- u_n(t; x,f_2) \Vert .
\end{align}
Integrating \eqref{vvv} over $[s , n]$ yields
\begin{align}\nonumber
a(s)\frac{d}{ds}&{\Vert u_n(s;x,f_1) - u_n(s; x,f_2)\Vert}^2  \\
\nonumber \le \ &2\int_s^n b\Vert f_1-f_2\Vert \cdot \Vert
u_n(\tau;x,f_1) - u_n(\tau; x,f_2) \Vert \, d\tau .
\end{align}
A new integration, this time over $[0,t]$, leads to
\begin{align}\nonumber
\int_0^ta(s)&\frac{d}{ds}{\Vert u_n(s;x,f_1) - u_n(s;
x,f_2)\Vert}^2ds   \\ \nonumber \le & \ 2\int_0^n ds \int_s^n
b\Vert f_1-f_2\Vert \cdot \Vert u_n(\tau;x,f_1) - u_n(\tau; x,f_2)
\Vert \, d\tau \\ \nonumber = & \ 2\int_0^n sb\Vert f_1-f_2\Vert
\cdot \Vert u_n(s;x,f_1) - u_n(s; x,f_2) \Vert \, ds.
\end{align}
Therefore,
\begin{align}\label{ss}\nonumber
a(t)&{\Vert u_n(t;x,f_1) - u_n(t; x,f_2)\Vert}^2 \\
\le & \ C_n + \int_0^t\frac{q^+(s)}{p(s)}a(s){\Vert u_n(s;x,f_1) -
u_n(s; x,f_2)\Vert}^2ds,
\end{align}
where
$$
C_n = 2\int_0^n sb\Vert f_1-f_2\Vert \cdot \Vert u_n(s;x,f_1) -
u_n(s; x,f_2) \Vert \, ds.
$$
Using the Gronwall-Bellman lemma, we derive from \eqref{ss}
\begin{align}\nonumber
a_-(t)&{\Vert u_n(t;x,f_1) - u_n(t; x,f_2)\Vert}^2 \le  \ C_n
\\ \nonumber \le & \ \frac{2}{p_0}a_+(\infty)\int_0^nsa_-(s)\Vert
f_1-f_2\Vert \cdot \Vert u_n(s;x,f_1) - u_n(s; x,f_2) \Vert \, ds.
\end{align}
Recall that $p_0 = \mathrm{ess} \inf \ p$. This implies
\begin{equation}
\sqrt{a_-(t)}\Vert u_n(t;x,f_1) - u_n(t; x,f_2)\Vert  \le \
C\int_0^ns\sqrt{a_-(s)}\Vert f_1-f_2\Vert \, ds,
\end{equation}
where $C=2a_+(\infty )/2$. This leads to
\begin{align}\label{doi}\nonumber
\sqrt{a_-(t)}\Vert u(t;x,f_1) - u(t; x,f_2)\Vert  \le & \ C\int_0^{\infty }s\sqrt{a_-(s)}\Vert f_1-f_2\Vert \, ds \\
= & C{\Vert f_1 - f_2  \Vert}_Y, \ \ \forall t\ge 0.
\end{align}
From inequality \eqref{doi} we can easily derive the existence of
a unique weak solution u(t;x,f) for each $x\in \overline{D(A)}$
and $f\in X$. Indeed, it is sufficient to observe that $f$ can be
approximated (with respect to the norm of $Y$) by a sequence
$(f_k)$ of smooth functions with compact support $\subset
(0,\infty)$ and use \eqref{doi} with $f_1:=f_k$ and
$f_2:=f_j$. Note that \eqref{sasa} holds for
$u_n^{\prime}(t;x,f_k)$ with $E(x,f_k)$ (which is also bounded),
so \eqref{mmm}  also holds true for $u^{\prime}(t;x,f_k)$.
Therefore, $\sqrt{ta_-}u^{\prime} \in L^2({\mathbb{R}}_+;H)$ (as the
weak limit in $L^2({\mathbb{R}}_+;H)$ of the sequence
$(\sqrt{ta_-}u_k^{\prime})$). This completes the proof of the
theorem.
\end{proof}

\begin{remark}
If we assume the stronger
condition $q\in L^1({\mathbb{R}}_+)$, then condition \eqref{C} 
becomes
\begin{equation}\tag{$C1$}\label{C1}
\sup_{t\ge 0}\Vert u(t) \Vert < \infty .
\end{equation}
In this case we can state the following result
\end{remark}
\begin{corollary}\label{corollary}
Assume that $(H1)$ holds, $p \in L^{\infty}({\mathbb{R}}_{+})$,
with $\mathrm{ess} \inf p>0$, and $q \in
L^{\infty}({\mathbb{R}}_{+})\cap L^1({\mathbb{R}}_{+})$. Then, for
each $x\in \overline{D(A)}$ and $f\in Y:=L^1(0,\infty ;H; tdt )$,
there exists a unique weak solution $u$ of problem \eqref{E},
\eqref{B}, \eqref{C1}, satisfying $t^{1/2}u^{\prime} \in
L^2({\mathbb{R}}_+;H)$. If in addition $f\in L^2_{loc}([0,\infty
);H)$, then $u$ is a strong solution satisfying $t^{3/2}u^{\prime
\prime}\in L^2_{loc}([0,\infty );H)$. If further $x\in D(A)$, then
$u\in W^{2,2}_{loc}([0, \infty );H)$.
\end{corollary}

\begin{proof}
By the uniqueness property it follows that every strong solution 
of problem \eqref{E}, \eqref{B}, \eqref{C1}, associated with $(x,f)\in D(A)\times \big[
X\cap L^2_{loc}(0,\infty ;H) \big]$, denoted $u(t;x,f)$, can be
obtained by the limiting procedure developed in the proof of
Theorem \ref{t1}. By \eqref{kj} and \eqref{doi}, we have
\begin{align}\label{basic} \nonumber
\Vert u(t;x_1,f_1) - u(t; x_2,f_2)\Vert  \le & \Vert u(t;x_1,f_1)
- u(t; x_2,f_1)\Vert +
\Vert u(t;x_2,f_1) - u(t; x_2,f_2)\Vert \\
\le & \Vert x_1 - x_2 \Vert + \hat{C}{\Vert f_1 - f_2 \Vert}_Y \ \
\forall t\ge 0,
\end{align}
which is valid for all $(x_i, f_i)\in \overline{D(A)}\times \big[
Y\cap L^2_{loc}(0,\infty ;H) \big]$, $i=1,2$ (i.e., for strong
solutions), and can be extended to all $(x_i, f_i)\in
\overline{D(A)}\times Y$. Inequality \eqref{basic} shows that for
each pair $(x,f)\in \overline{D(A)} \times Y$ there exists a
unique weak solution $u(t;x,f)$ of problem \eqref{E}, \eqref{B},
\eqref{C1}. The rest of the proof follows easily from Theorems
\ref{t1} and \ref{t2}.
\end{proof}

\section{Some Comments}

a. {\bf  Regularity of weak solutions}.  
If $u=u(t;x,f)$ is the weak solution given by Theorem \ref{t2} 
corresponding to some $(x,f)\in \overline{D(A)}\times Y$ and 
$f\in L^2_{loc}([t_0, \infty );H)$ for some $t_0 >0$, then $u$ 
is a strong solution on $[t_0, \infty )$ (i.e., a weak solution 
becomes strong once $f$ becomes locally square integrable). The 
proof of this assertion follows by the uniqueness property of 
weak solutions.
\vskip10pt
b. {\bf The condition $f\in Y$ is optimal in the results above}.
As in \cite{KM}, we consider the simple example $H=\mathbb{R}$,
$A=0$, $p\equiv 1$, $q\equiv 0$, $f(t)={(t+1)}^{-2 - \delta}, \
\delta >0$. For all $x\in \mathbb{R}$, problem \eqref{E},
\eqref{B}, \eqref{C} has a unique solution. Obviously, $f\in Y$
and the result is in line with Corollary \ref{corollary}. If $\delta =0$
then $f$ is no longer a member of $Y$ and all solutions of
equation \eqref{E} are unbounded. 
\vskip10pt 
c. {\bf The
contraction semigroup generated by $u(t;x,0), \ x\in
\overline{D(A)}$}. For all $x\in \overline{D(A)}$, define
$S(t)x:=u(t;x,0), \ t\ge 0$, where $u$ is the solution given by
Theorem \ref{t1}. Then, according to \eqref{basic}, the family $\{
S(t): \overline{D(A)} \rightarrow \overline{D(A)} \}$ is a
semigroup of contractions. In the special case $p\equiv 1$ and
$q\equiv 0$, the infinitesimal generator of this semigroup is
precisely the square root $A^{1/2}$ of $A$, as defined by V. Barbu
(see Chapter V of \cite{VB3} and \cite{HB1} for details on this
semigroup and its generator). 
\vskip10pt 
d. {\bf Smoothing effect
on the starting values}. Let $f\in Y\cap L^2_{loc}([0, \infty );H)$. Then, the
solution $u(t;x,f)$ starting from $x\in \overline{D(A)}$ is a
strong one, so in particular $u(t;x,f) \in D(A)$ for a.a. $t>0$.
This is a smoothing effect: if, for example, $A$ is a partial
differential operator, then $D(A)$ contains functions which are
more regular than those in $\overline{D(A)}$. In the case when $p,
\, q, \, f$ are smooth functions, it is expected that for any
$x\in \overline{D(A)}$, $u(t;x,f) \in D(A)$ for all $t>0$ and
that $u(t;x,f)$ satisfy equation \eqref{E} for all $t>0$ (not just
for a.a. $t>0$). In the special case $p\equiv 1$, $q \equiv 0$,
$f\equiv 0$, this does happen (see \cite{VB3}, p. 315). In this
case $u(t;x,0)=S(t)x$, where $\{ S(t); \, t\ge 0 \}$ is the
semigroup generated by $A^{1/2}$, which is a nice operator, and
$u(t;x,0)$ is the solution of $u^{\prime} + A^{1/2}u \ni 0$.
\vskip10pt 
e. {\bf Variational approach}. Assume that $A$ is the
subdifferential of a proper, convex, lower semicontinuous function
$\phi : H \rightarrow (-\infty , +\infty]$. Since the graph of $A$
contains $[0,0]$, one can assume that $0=\phi (0) = \min \{ \phi
(z): \, z\in H \}$. Recall that for all $(x,f)\in D(A)\times \Big[
 Y \cap L^2_{loc}([0, \infty ); H) \Big]$ the solution $u =
u(t;x,f)$ given by Theorem \ref{t1} on a given interval $[0,R]$ is
the limit of $(u_n)$ in $C([0,R];H)$. Since $u_n$ is the solution
of the two-point boundary value problem \eqref{un}, \eqref{BCun},
it is the minimizer of the functional $\Psi_n : L^2(0,n;H)
\rightarrow (-\infty , +\infty ]$ defined by  $\Psi_n (v) =
\frac{1}{2}\int_0^n\Big( a {\Vert v^{\prime}  \Vert}^2dt  + b\phi
(v) + b(f,v)\Big)\, dt$, if $v\in W^{1,2}(0,n;H), \ \phi (v)\in
L^1(0,n), \ v(0)=x, \ v(n)=0$, and  $\Psi_n (v) =+\infty $,
otherwise.

In fact, any (weak) solution $u(t;x,f)$, $(x,f)\in \overline{D(A)}
\times Y$, can be approximated on compact intervals by minimizers
$u_n$ associated with $(\bar{x}, \bar{f}) \in D(A)\times \Big[ Y
\cap L^2_{loc}([0, \infty ); H) \Big]$ close to $(x,f)$ in
$H\times Y$.

\section{Approximation by the method of artificial viscosity}
Let $\varepsilon >0$ be a small number and let $p(t)=\varepsilon , \ \forall t\ge 0$. In this case, equation \eqref{E} can be regarded as an approximate one for the following reduced equation
\begin{equation}\tag{$E_0$}\label{E_0}
q(t)u^{\prime}(t)\in Au(t)+f(t) \quad \text{for a.a. }t >0.
\end{equation}
Equation \eqref{E_0} where $q^+ \equiv 0$ (i.e., $q(t)\le 0$ for
a.a. $t>0$) is particularly significant for applications to
parabolic and hyperbolic PDE problems, as explained below. It is
expected that any solution $u_{\varepsilon}(t;x,f)$ of \eqref{E}
(with $p\equiv \varepsilon $), \eqref{B}, satisfying
$\sqrt{a_-}\Vert u_{\varepsilon}(\cdot \, ;x,f) \Vert \in
L^{\infty}({\mathbb{R}}_+)$, approximate in some sense the
solution $u(\cdot \, ;x,f)$ of \eqref{E_0}, \eqref{B}, for
$\varepsilon $ small enough. The advantage is that
$u_{\varepsilon}$ is more regular (with respect to $t$) than $u$.
This method of approximation (called the method of artificial
viscosity, due to the term involving $\varepsilon$ in \eqref{E})
was introduced and studied by J.L. Lions \cite{JLL} mainly in
the case of linear PDE problems. See also \cite{BM}. Here we have
more general problems on the whole positive half line that require
separate analysis. Hopefully some results on this subject will be
obtained later. In the following we just present some examples
which seem suitable for the artificial viscosity method.
\vskip10pt Let $\Omega$ be a bounded open subset of
${\mathbb{R}}^k$ with a smooth boundary $\Gamma $. Let $\beta
:D(\beta)\subset \mathbb{R} \rightarrow \mathbb{R}$ be a (possibly
set-valued) maximal monotone mapping, with $0 \in D(\beta)$ and
$0\in \beta (0)$. Consider the nonlinear {\bf diffusion-reaction
equation}
\begin{equation}\tag{$E_1$}\label{E_1}
u_t - div \, \Big( r(x)\, grad \, u \Big) + \beta (u) \ni f(t,x),
\ \ (t,x) \in (0, \infty )\times \Omega ,
\end{equation}
with the Dirichlet boundary condition
\begin{equation}\tag{$DBC$}\label{DBC}
u = 0 \ \ \text{on } (0, \infty ) \times \Gamma,
\end{equation}
and the initial condition
\begin{equation}\nonumber
u(0,x) = u_0(x), \ \ x\in \Omega.
\end{equation}
The function $r=r(x)$ in \eqref{E_1} is assumed to be a
nonnegative smooth function. Obviously, \eqref{E_1}, \eqref{DBC}
can be expressed as an equation of the form \eqref{E_0} in
$H=L^2(\Omega )$ with $q\equiv -1$, and $A$ a maximal monotone
operator in $H$. The corresponding approximate equation (i.e., 
Eq. \eqref{E} with $p\equiv \varepsilon $) is 
\begin{equation}\nonumber
\varepsilon u_{tt} - u_t \in - div \ \Big( r(x)\, grad \, u \Big)
+ \beta (u) + f(t,x), \ \ (t,x) \in (0, \infty )\times \Omega ,
\end{equation}
with the same boundary condition \eqref{DBC}. Note that this is an elliptic type equation with respect to $(t,x)=(t, x_1,...,x_k)$.
\vskip10pt
The nonlinear {\bf wave equation}
\begin{equation}\nonumber
u_{tt} -\Delta u +\beta (u_t) \ni f(t,x), \ \ (t,x) \in (0, \infty
)\times \Omega ,
\end{equation}
with \eqref{DBC}, could be also examined. It is well known that this equation can be represented (by using the substitution $v=u_t$) as an equation of the form \eqref{E_0} with $q \equiv -1$ in the product (phase) space $H=H_0^1(\Omega ) \times L^2(\Omega )$ (see, e.g., \cite{GM}, p. 205). It is easily seen that the approximate equation (i.e., \eqref{E} with $p \equiv \varepsilon $), associated with the wave equation above, is equivalent to 
\begin{equation}\nonumber
{\varepsilon}^2u_{tttt} -2\varepsilon u_{ttt} + u_{tt}-\Delta u + \beta (u_t - \varepsilon u_{tt})\ni f(t,x),
\end{equation}
which obviously provides solutions which are more regular (with respect to $t$) than those of the wave equation.

\end{document}